\documentclass[a4paper]{article}
\usepackage{amsmath,amsfonts,amssymb,amsthm,mathtools,pstricks,caption}

\newtheorem{lemma}{Lemma}
\newtheorem{theorem}[lemma]{Theorem}
\newtheorem{proposition}[lemma]{Proposition}
\newtheorem{corollary}[lemma]{Corollary}

\theoremstyle{definition}
\newtheorem{definition}[lemma]{Definition}
\newtheorem{remark}[lemma]{Remark}

\newcounter{step}[section]
\newcommand{\step}{\refstepcounter{step}\paragraph{Step \arabic{step}:}}

\newcommand{\Sphere}{\mathbb{S}^3}
\DeclareMathOperator{\Int}{int}
\newcommand{\nhd}{\mathcal{N}}

\begin{document}
\title{Genus and fibredness of certain three-bridge links and certain satellite knots}
\author{Jessica E. Banks}
\date{}%17 Oct 2014
\maketitle
\begin{abstract} 
For each three-bridge link of a certain form, we construct a taut Seifert surface for the link and establish whether the link is fibred. Using this, we also give the genus and fibredness of satellite knots whose pattern is constructed from a two-component two-bridge link in the case not addressed by work of Hirasawa and Murasugi.
\end{abstract}

%------------------

\section{Introduction}
%master document is threebridge.tex

Two-bridge knots and links (also known as rational links) have been extensively studied (see, for example, \cite{MR2028021}, \cite{MR778125}, \cite{MR2218754}, \cite{MR1807953}, \cite{MR2948663}, \cite{MR2141874}). In particular, \cite{MR808776} Proposition 12.24 gives the genus of any two-bridge link and a necessary and sufficient condition for the link to be fibred. Three-bridge links have also been studied (see, for example, \cite{MR779063}, \cite{MR3007934}, \cite{MR2493373}, \cite{MR2809049}, \cite{MR1714299}), but to a far lesser degree. In this paper, we consider the family of three-bridge links given by doubling one component of a two-component two-bridge link with a particular orientation. Using sutured manifold techniques developed by Gabai (see, for example, \cite{MR723813}, \cite{MR870705}, \cite{MR823442}), we give an explicit construction of a taut Seifert surface for each of these links and calculate when each link is fibred.

In \cite{zbMATH05545482} it is shown that three-bridge links of the form just described can be used to understand (in particular) certain satellite knots that are constructed using two-bridge links. For most such satellite knots, Hirasawa and Murasugi, in \cite{zbMATH05820036}, calculate the genus of the knot and whether it is fibred. There is one case that is not covered by their analysis. By combining Proposition \ref{tautnessproofprop} and Corollary \ref{whenfibredcor} with Theorems 1 and 2 of \cite{zbMATH05545482}, we answer these questions in the remaining case. We include a discussion of how these results can be understood in terms of sutured manifolds.

A \textit{Seifert surface} for a link $K\subseteq \Sphere$ is an oriented surface $R_K$, with no closed components, embedded in $\Sphere$ such that the oriented boundary of $R_K$ is $K$. We say that $R_K$ is \textit{taut} if it has maximal Euler characteristic among Seifert surfaces for $K$. The \textit{genus} of a knot is the genus of a taut Seifert surface.
A link is \textit{fibred} if its exterior is a fibre bundle over $\mathbb{S}^1$ with each fibre a (taut) Seifert surface.

In Section \ref{algorithmsection} we give an algorithm to construct a Seifert surface for each of the three-bridge links we are considering.
In Section \ref{decompositionssection} we recall various definitions and results regarding sutured manifolds we will need, and consider some specific sutured manifold decompositions.  
In Section \ref{tautnesssection}, we use sutured manifold decompositions to show that each of the constructed Seifert surfaces is taut, and in Section \ref{fibrednesssection} we further use these calculations to establish which of the links are fibred. As the proofs involve checking a number of similar cases, we postpone some of the details to the appendix.
The discussion of satellite knots appears in Section \ref{satellitesection}.

\section{An algorithm for constructing surfaces}\label{algorithmsection}
%master document is threebridge.tex

Take a four-string braid of the form $\sigma_2^{a_1}\sigma_3^{-a_2}\sigma_2^{a_3}\sigma_3^{-a_4}\cdots\sigma_3^{-a_{N-1}}\sigma_2^{a_N}$, for $N$ odd.
At the top and at the bottom of the braid, add two arcs, one joining the first and second strings, one joining the third and fourth strings. This gives (an alternating diagram of) a two-bridge knot or link $L_2$ with continued fraction $[a_1,a_2,a_3,\ldots,a_N]$. The corresponding fraction is
\[
\cfrac{1}{a_1+\cfrac{1}{a_2+\cfrac{1}{\stackrel{\ddots}{a_{N-1}+\cfrac{1}{a_N}.}}}}
\]
 Conversely, every two-bridge knot or link has a diagram of this form where either $a_i> 0$ for all $i$ or $a_i< 0$ for all $i$ (\cite{MR808776} Proposition 12.13). We will work from these diagrams, and assume that $a_i> 0$ for all $i$. For convenience, we will draw the diagram with the braid instead running from left to right, as shown in Figure \ref{algorithmpic1}. The boxes each represent a line of crossings, the direction of which is given by the crossing shown in the box.
\begin{figure}[htbp]
\centering
\input{picturefiles/algorithmpic1.tex}
\caption{\label{algorithmpic1}}
\end{figure}

For this paper we are interested only in two-component two-bridge links. In this case, one component is made up of the first string of the braid together with the arc that is the second string at both the top and the bottom  of the braid (and the two pieces of arc joining these two strings). This means that the diagram of this component is an embedded loop in the plane. Call this component $L_{\mathbb{A}}$, and the other component $L_{\mathbb{D}}$. We will double $L_{\mathbb{A}}$, taking a parallel copy of it with some framing. The orientation on the new component is chosen to be the opposite of that on $L_{\mathbb{A}}$, so that the two parallel components of the new link bound an annulus in the complement of the third link component; call this annulus ${\mathbb{A}}$. An example is shown in Figure \ref{algorithmpic2} (in this case with the framing given by the diagram). Call the new link $L_3$.  The class of links we are interested in is those given by this construction. Note that each is a three-bridge link --- with three components it cannot have a bridge number below $3$, and we have drawn a diagram demonstrating that the bridge number is at most three.
Denote by $L_+$ the component of $\partial {\mathbb{A}}$ that is oriented from right to left along the first string of the original braid (and therefore runs from left to right through the main section of the braid), and denote the other component of $\partial {\mathbb{A}}$ by $L_-$. We will continue to call the third component $L_{\mathbb{D}}$.
Note also that $L_3$ is not split unless $L_2$ is split (in which case $L_2$ is an unlink, and $L_3$ is either an unlink or an unknot together with a $(2,2m)$ torus link), since any separating sphere would necessarily separate $L_{\mathbb{D}}$ from at least one of $L_+$ and $L_-$.
\begin{figure}[htbp]
\centering
\input{picturefiles/algorithmpic2.tex}
\caption{\label{algorithmpic2}}
\end{figure}

In general $L_{\mathbb{D}}$ is not embedded in the given diagram. However, it appears as a two-string braid (with an arc joining the strings at the top and at the bottom). It therefore bounds a disc ${\mathbb{D}}$ that is (nearly always) divided into two monogons connected by a line of bigons (as in Figure \ref{algorithmpic3}). Note that the directions of the crossings in the diagram of $L_{\mathbb{D}}$ are not necessarily all the same. In fact, the direction of each crossing is determined by the position of $L_{\mathbb{A}}$ at the corresponding point of the braid.
\begin{figure}[htbp]
\centering
\input{picturefiles/algorithmpic3.tex}
\caption{\label{algorithmpic3}}
\end{figure}

For notational purposes, take a point of $L_{\mathbb{A}}$ on the first string of the braid and isotope it to infinity in the diagram plane. This gives a picture of $L_2$ that we can see as the disc ${\mathbb{D}}$ with $L_{\mathbb{A}}$ as an infinite arc running from left to right that winds up and down across ${\mathbb{D}}$ in the process (see Figure \ref{algorithmpic4} for an example). 
We will divide up this picture into vertical strips, with each piece drawn from a fixed list. The arrangement of these pieces will help determine a Seifert surface for $L_3$. 
\begin{figure}[htbp]
\centering
\input{picturefiles/algorithmpic4.tex}
\caption{\label{algorithmpic4}}
\end{figure}

\step
Each section of $L_{\mathbb{A}}$ that runs over the disc ${\mathbb{D}}$ takes one of the four forms shown in Figure \ref{algorithmpic5}a. Label each of these with a letter $\mathcal{A}$, $\mathcal{B}$, $\mathcal{C}$ or $\mathcal{D}$ respectively. 
\begin{figure}[htb]
\centering
(a)
\input{picturefiles/algorithmpic5a.tex}\\\mbox{} \\
(b)
\input{picturefiles/algorithmpic5b.tex}
(c)
\input{picturefiles/algorithmpic5c.tex}
(d)
\input{picturefiles/algorithmpic5d.tex}
\caption{\label{algorithmpic5}}
\end{figure}
Between each consecutive pair of these is one of the blocks shown in Figure \ref{algorithmpic5}b. Again the box represents a non-negative number of crossings in the direction indicated (here, unlike previously, we allow this number to be $0$). Label each of these blocks by either an $\mathcal{O}$ or an $\mathcal{E}$ according to whether the number of crossings it represents is odd or even. Reading from left to right, this gives a word in the alphabet $\{\mathcal{A},\mathcal{B},\mathcal{C},\mathcal{D},\mathcal{O},\mathcal{E}\}$. 
An example of this subdivision is shown in Figure \ref{algorithmpic6}; the resulting word is $\mathcal{A}\mathcal{E}\mathcal{C}\mathcal{E}\mathcal{B}\mathcal{O}\mathcal{B}
\mathcal{E}\mathcal{D}\mathcal{E}\mathcal{A}\mathcal{O}\mathcal{C}\mathcal{O}
\mathcal{D}\mathcal{O}\mathcal{A}$.
\begin{figure}[htbp]
\centering
\input{picturefiles/algorithmpic6.tex}
\caption{\label{algorithmpic6}}
\end{figure}
At the ends of the picture are the blocks shown in Figure \ref{algorithmpic5}c. We will actually treat these two pieces as a single block as in Figure \ref{algorithmpic5}d, treating the whole picture as circular rather than linear. Assign this final block either an $\mathcal{O}$ or an $\mathcal{E}$ so that the total number of instances of $\mathcal{O}$ is even. For example, given the word $\mathcal{A}\mathcal{E}\mathcal{C}\mathcal{E}\mathcal{B}\mathcal{O}\mathcal{B}
\mathcal{E}\mathcal{D}\mathcal{E}\mathcal{A}\mathcal{O}\mathcal{C}\mathcal{O}
\mathcal{D}\mathcal{O}\mathcal{A}$, we assign an $\mathcal{E}$ to this block. Placing this letter at the start/end of the word we already had results in a circular word. Call this word $\mathcal{W}$. We will call the letter corresponding to the ends of $\mathbb{D}$ the \textit{infinity letter} (because it corresponds to a tangle with slope $\infty$, whereas each other $\mathcal{O}$ or $\mathcal{E}$ corresponds to a tangle with integer slope). 

Which of the blocks from Figure \ref{algorithmpic5}a occur (or, whether $L_{\mathbb{A}}$ lies above or below the disc ${\mathbb{D}}$) will not play an essential role in the proofs in this paper; this variation simply increases the number of cases to be considered. For this reason, we will assume from now on that only the $\mathcal{A}$ block in Figure \ref{algorithmpic5}a occurs. The details for the other cases can be found in the appendix. Figure \ref{algorithmpic7} shows a link of this form that otherwise follows the pattern of the link in Figure \ref{algorithmpic6}.
\begin{figure}[htbp]
\centering
\input{picturefiles/algorithmpic7.tex}
\caption{\label{algorithmpic7}}
\end{figure}

\step
The next step is to place either a $+$ or a $-$ between each pair of letters in $\mathcal{W}$. 
First place a $+$ to the right side of the infinity letter, and then follow the following rules.
\begin{itemize}
\item An $\mathcal{A}$ (or $\mathcal{B}$, $\mathcal{C}$, $\mathcal{D}$) should have the same symbol ($+$ or $-$) on either side.
\item An $\mathcal{E}$ should have the same symbol on either side.
\item An $\mathcal{O}$ should have different symbols on the two sides.
\end{itemize}
For example, the link in Figure \ref{algorithmpic7} gives us
\[
\mathcal{E}+\mathcal{A}+\mathcal{E}+\mathcal{A}+\mathcal{E}+\mathcal{A}+\mathcal{O}-\mathcal{A}-
\mathcal{E}-\mathcal{A}-\mathcal{E}-\mathcal{A}-\mathcal{O}+\mathcal{A}+\mathcal{O}-
\mathcal{A}-\mathcal{O}+\mathcal{A}+.
\]

The significance of these signs is that they tell us about the orientation of the disc $\mathbb{D}$ at the relevant points in the sequence. Without loss of generality, we may assume that the leftmost monogon of $\mathbb{D}$ (which makes up the right-hand part of the block from Figure \ref{algorithmpic5}d) is oriented pointing upwards, corresponding to a $+$ symbol. Each letter $\mathcal{E}$ corresponds to an even number of half-twists in the disc $\mathbb{D}$, so (looking at Figure \ref{algorithmpic5}b) the orientation of $\mathbb{D}$ on the left-hand side of the block matches that on the right-hand side. On the other hand, an $\mathcal{O}$ corresponds to an odd number of half-twists, so the orientations of $\mathbb{D}$ on the left-hand and right-hand sides of the corresponding block are different.

\step\label{pairingstep} 
We will pair up many of the $\mathcal{A}$ letters in $\mathcal{W}$. To explain the reasoning behind this pairing, we will first study some smaller examples.

First consider the example shown in Figure \ref{algorithmpic8}a. Here $\mathbb{D}$ is split into two monogons, one oriented upwards and the other downwards. Since $L_{\mathbb{A}}$ passes through each of these discs from above to below, tubing together the two monogons (as shown in Figure \ref{algorithmpic8}b) results in an orientable surface $\mathbb{D}'$ (an annulus) disjoint from $L_{\mathbb{A}}$. Figure \ref{algorithmpic8}c gives an alternative picture of this. We may now double $L_{\mathbb{A}}$ and insert the annulus ${\mathbb{A}}$ in the complement of $\mathbb{D}'$, giving a (disconnected) Seifert surface for $L_3$ as shown in Figure \ref{algorithmpic8}d. 
\begin{figure}[htbp]
\centering
(a)
\input{picturefiles/algorithmpic8a.tex}
(b)
\input{picturefiles/algorithmpic8b.tex}\\
(c)
\input{picturefiles/algorithmpic8c.tex}
(d)
\input{picturefiles/algorithmpic8d.tex}
\caption{\label{algorithmpic8}}
\end{figure}

Now take the example shown in Figure \ref{algorithmpic9}a. The presence of a letter $\mathcal{O}$ (in the middle of the word) again allows us to tube together two points on the disc $\mathbb{D}$ to reduce the intersection of $\mathbb{D}$ and $L_{\mathbb{A}}$ and still have an orientable surface. 
Any time there is a letter $\mathcal{O}$ in $\mathcal{W}$ we can tube together the pieces of disc coming from the two copies of the block $\mathcal{A}$ on either side. However, this can only be done once with each $\mathcal{A}$; if the letter $\mathcal{O}$ occurs twice in a row (that is, with no $\mathcal{O}$ or $\mathcal{E}$ between them), we can only apply this trick for one of these two letters.
As the infinity letter is also an $\mathcal{O}$, here we could tube together the two monogons to remove the remaining intersections between the surface and $L_{\mathbb{A}}$. In this sense the infinity letter behaves the same way as the other $\mathcal{O}$ in $\mathcal{W}$.  For illustrative purposes, however, we will ignore this option. Instead we note that the two remaining copies of $\mathcal{A}$ (the outer two) appear as one $+\mathcal{A} +$ and one $-\mathcal{A} -$. This means we can tube these two pieces of disc together through the middle of the previous tube to again give an orientable surface (as shown in Figure \ref{algorithmpic9}b). 
\begin{figure}[htbp]
\centering
(a)
\input{picturefiles/algorithmpic9a.tex}\\
(b)
\input{picturefiles/algorithmpic9b.tex}
\caption{\label{algorithmpic9}}
\end{figure}

This second example demonstrates the aim of our pairing. The change in sign means the orientability of the surface is preserved. The fact that we have already tubed together the two $\mathcal{A}$ blocks in between means there is space to add the second tube. In general we wish to tube together the pieces of $\mathbb{D}$ corresponding to as many instances of the letter $\mathcal{A}$ as possible, while ensuring that we end up with an orientable surface. 
Therefore, we are looking to pair up the $\mathcal{A}$ letters in $\mathcal{W}$ obeying the following rules.
\begin{itemize}
\item Each letter $\mathcal{A}$ is paired with at most one other.
\item Each pair consists of one $+\mathcal{A} +$ and one $-\mathcal{A} -$.
\item No pairing interleaves with another pairing.
\item Given two copies of the letter $\mathcal{A}$ that are paired, these two letters divide the circular word $\mathcal{W}$ into two pieces. We require that in at least one of these pieces every letter $\mathcal{A}$ is paired. We call this piece the \textit{joining word} of the pair.
\item There are as many pairings as possible.
\end{itemize}
Let $N_+$ be the number of instances of $+\mathcal{A} +$ in $\mathcal{W}$, and $N_-$ the number of instances of $-\mathcal{A} -$. Without loss of generality we will assume that $N_+\geq N_-$.
From the first two conditions it is clear that we cannot pair up every letter $\mathcal{A}$ if $N_+> N_-$.

\begin{lemma}\label{pairinglemma}
It is possible to choose $N_-$ pairs meeting the above requirements. 
\end{lemma}
\begin{proof}
We proceed by induction in $N_-$. Clearly the result holds if $N_-=0$.
Suppose instead that $N_->0$.
Then $\mathcal{W}$ contains at least one letter $\mathcal{O}$.

First suppose there are two consecutive instances of the letter $\mathcal{O}$. Then there is a subword of $\mathcal{W}$ that is either $\mathcal{A}+\mathcal{O}-\mathcal{A}-\mathcal{O}+$ or $\mathcal{A}-\mathcal{O}+\mathcal{A}+\mathcal{O}-$. Pair these two instance of $\mathcal{A}$. Next let $\mathcal{W}'$ be the word given by deleting this string from $\mathcal{W}$. Inductively, there is a pairing for $\mathcal{W}'$ that includes each $-\mathcal{A}-$ and meets the above rules. Taking the same pairings on the letters in $\mathcal{W}$ completes the pairing process for $\mathcal{W}$.

Suppose instead that any two instances of the letter $\mathcal{O}$ in $\mathcal{W}$ have a letter $\mathcal{E}$ between them in each direction (recall that $\mathcal{W}$ is circular). As $\mathcal{W}$ contains the letter $\mathcal{O}$ at least twice, there is a subword of $\mathcal{W}$ that is $\mathcal{E}+\mathcal{A}+\mathcal{O}-\mathcal{A}-\mathcal{E}.$ Pair these two instances of the letter $\mathcal{A}$, and let $\mathcal{W}'$ be the word given by replacing this string in $\mathcal{W}$ with the single letter $\mathcal{O}$. Again, choosing a pairing on the instances of $\mathcal{A}$ in $\mathcal{W}'$ completes the pairing process on $\mathcal{W}$.
\end{proof}

Returning to our example in Figure \ref{algorithmpic7}, two possible pairings are
\[
\mathcal{E}+\mathcal{A}+\mathcal{E}+
\underbracket{\mathcal{A}+\mathcal{E}+\underbracket{\mathcal{A}+
\mathcal{O}-\mathcal{A}}-
\mathcal{E}-\mathcal{A}}-\mathcal{E}-
\underbracket{\mathcal{A}-\mathcal{O}+
\mathcal{A}}+\mathcal{O}-
\underbracket{\mathcal{A}-\mathcal{O}+\mathcal{A}}+
\]
and
\[
\mathcal{E}+\underbracket{\mathcal{A}+\mathcal{E}+
\underbracket{\mathcal{A}+\mathcal{E}+\mathcal{A}+
\mathcal{O}-\mathcal{A}}-
\mathcal{E}-\mathcal{A}}-\mathcal{E}-
\underbracket{\mathcal{A}-\mathcal{O}+
\underbracket{\mathcal{A}+\mathcal{O}-
\mathcal{A}}-\mathcal{O}+\mathcal{A}}+.
\]

\step\label{tubingstep}
Given a pairing as in the previous step, we can now define a Seifert surface $R$ for $L_3$.
Start with the disc $\mathbb{D}$. There is a partial order on the pairs of letters given by inclusion of the joining words.
In our two example pairings above, in the first case only two of the four joining words are comparable, whereas in the second case the words are totally ordered.
Tube together the pieces of $\mathbb{D}$ corresponding to the paired letters in turn, respecting this partial order. That is, first add tubes for pairs of the letter $\mathcal{A}$ that are close together, and then move on to those that are further apart. In each case the tube should follow $L_{\mathbb{A}}$ in the direction of the joining word for the pair. Thus each tube will pass through the tubes corresponding to joining words that are lower in the partial order.
The result of this will be an orientable, once-punctured surface $\mathbb{D}'$ of genus $N_-$ with boundary equal to $L_{\mathbb{D}}$.

If $N_+=N_-$ then $\mathbb{D}'$ is disjoint from $L_{\mathbb{A}}$. We may therefore double $L_{\mathbb{A}}$ and add in the annulus $\mathbb{A}$ in the complement of $\mathbb{D}'$, giving a disconnected Seifert surface $R=\mathbb{D}'\cup \mathbb{A}$ for $L_3$.

If $N_+>N_-$ then there will be $N_+-N_-$ instances of the letter $\mathcal{A}$ in $\mathcal{W}$ that are not paired. Each of these will appear in $\mathcal{W}$ as $+\mathcal{A} +$.
Again, double $L_{\mathbb{A}}$ to give $L_+$ and $L_-$, and add in the annulus $\mathbb{A}$. This time $\mathbb{A}$ will intersect $\mathbb{D}'$ in $N_+$ arcs, one in each unpaired $\mathcal{A}$ block.
Position $\mathbb{A}$ so that the number of crossings between $L_+$ and $L_-$ is minimal, and all such crossings occur  underneath the part of $\mathbb{D}$ corresponding to the infinity letter. That is, all these crossings should occur long the piece of $L_{\mathbb{A}}$ in the single block from Figure \ref{algorithmpic5}d.
By `turning over' $\mathbb{A}$ if needed, further ensure that, away from the infinity letter, $\mathbb{A}$ is oriented downwards (that is, $L_+$ appears above $L_-$ in the diagram).
Then each block corresponding to an unpaired $\mathcal{A}$ appears as in Figure \ref{algorithmpic10}a. We replace each of these with the piece of surface shown in Figure \ref{algorithmpic10}b. This results in an embedded, connected, orientable Seifert surface $R$ for $L_3$. Note that the orientations are important in Figure \ref{algorithmpic10}b, and therefore in this construction the orientability of $R$ is reliant on the fact that each unpaired $\mathcal{A}$ appears as $+\mathcal{A} +$. If we had instead found that $N_+<N_-$ we would have needed to position $\mathbb{A}$ the other way up.
\begin{figure}[htbp]
\centering
(a)
\input{picturefiles/algorithmpic10a.tex}
(b)
\input{picturefiles/algorithmpic10b.tex}
\caption{\label{algorithmpic10}}
\end{figure}

\begin{remark}\label{genusremark}
The Euler characteristic of $R$ is
\[\chi(R)=1+0-2N_--2(N_+-N_-)=1-2N_+.\]
As $L_3$ has three components, this means $R$ has genus $N_+-1$ if $N_-<N_+$ (when $R$ is connected), or genus $N_+$ if $N_-=N_+$ (when $R$ has two components, one of which is the annulus $\mathbb{A}$).
\end{remark}

\begin{remark}
There is one situation not covered by the above constructions --- when the word $\mathcal{W}$ is the empty word. Since the length of $\mathcal{W}$ is proportional to the number of blocks from Figure \ref{algorithmpic5}a that appear, this can only occur when $L_{\mathbb{A}}$ is disjoint from the disc $\mathbb{D}$ in the diagram of $L_2$. In this case, $L_2$ is the two-component unlink and $L_3$ bounds either three discs or a disc and an annulus. We will ignore this situation in this paper.
\end{remark}

\section{Sutured manifolds}\label{decompositionssection}
%master file threebridge.tex

\begin{definition}
A \textit{sutured manifold} is a compact $3$--manifold $M$ with a division of $\partial M$ into three oriented subsurfaces $T$, $R_+$, and $R_-$ such that:
\begin{itemize}
\item $T\cup R_+\cup R_-=\partial M$;
\item $T$ is a collection of tori;
\item $T\cap(R_+\cup R_-)=\emptyset$;
\item $s=R_+\cap R_-$ is a finite collection of disjoint simple closed curves in $\partial M$;
\item each component of $s$ inherits the same orientation from $R_+$ as from $R_-$.
\end{itemize}
The curves in $s$ are called the \textit{sutures}.
\end{definition}

 We will retain this notation throughout. Note that the orientation of $\partial M$ changes on crossing $s$. We assume that $R_+$ is oriented pointing out of $M$, and $R_-$ is oriented pointing in.
In this paper we will only consider the case $T=\emptyset$. If we are not concerned with which of $R_+$ and $R_-$ is which, it suffices to give a suitable set $s$ of (unoriented) sutures on $\partial M$ in order to make $M$ into a sutured manifold. We will denote the sutured manifold by either $(M,R_+,R_-)$ or $(M,s)$.

\begin{definition}
A \textit{product sutured manifold} is a sutured manifold $(M,s)$ such that there exists a surface $S$ and a homeomorphism $\psi\colon (M,s)\to (S\times[0,1],\partial S\times\{\frac{1}{2}\})$.
\end{definition}

\begin{remark}
If $(M,s)$ is a connected product sutured manifold then $R_+$ and $R_-$ are both connected.
\end{remark}

\begin{definition}
Let $S$ be a compact surface. If $S$ is connected then the \textit{Thurston norm} of $S$ is $\chi_-(S)=\min(0,-\chi(S))$, where $\chi(S)$ is the Euler characteristic of $S$. If $S$ is disconnected, the \textit{Thurston norm} of $S$ is given by summing the Thurston norms of its connected components.

In other words, the Thurston norm is the absolute value of the Euler characteristic after all sphere and disc components have been discarded.
\end{definition}

\begin{definition}
Let $(M,s)$ be a sutured manifold, and let $S$ be a surface either properly embedded in $M$ or contained in $\partial M$. Suppose that $\partial S=s$. Then $S$ is said to be \textit{taut} if $\chi_-(S)$ is minimal among all representatives of $[S,\partial S]\in H_2(M,s)$.

A sutured manifold $(M,s)$ is \textit{taut} if $M$ is irreducible and both $R_+$ and $R_-$ are taut.
\end{definition}

\begin{remark}
Any product sutured manifold is taut. In particular, a ball with a single suture is taut.
\end{remark}

\begin{definition}
Let $(M,R_+,R_-)=(M,s)$ be a sutured manifold.
A \textit{product disc} is a disc $S$ properly embedded in $M$ with $|\partial S\cap s|=2$.
A \textit{product annulus} is an annulus properly embedded in $M$ with one boundary component contained within $R_+$ and the other contained within $R_-$.
\end{definition}

\begin{definition}
Let $S\subseteq M$ be an orientable surface properly embedded in a sutured manifold $(M,R_+,R_-)$, with $\partial S$ transverse to $s$. Performing a \textit{sutured manifold decomposition} along $S$ gives a new sutured manifold $(M',R'_+,R'_-)$ as follows.
This is denoted by $(M,R_+,R_-)\rightsquigarrow(M',R'_+,R'_-)$.

Choose an orientation on $S$, and choose a product neighbourhood $S\times[0,1]$ of $S$ in $M$ such that $S$ is oriented towards $S\times\{1\}$.
Set $M'=M\setminus S\times(0,1)$. In addition, set $R'_+=(R_+\cap M')\cup (S\times\{1\})$ and $R'_-=(R_-\cap M')\cup (S\times\{0\})$.
\end{definition}

This definition means that in general there is a choice to be made, when decomposing, about the orientation of the decomposing surface $S$. For a product disc decomposition this choice has no effect up to isotopy of the sutures in $\partial M$.

\begin{remark}\label{twostepremark}
We could instead see a sutured manifold decomposition as a two step process: first alter the sutures according to set rules, given $S$, to make them disjoint from $\partial S$, then delete (the interior of) a product neighbourhood of $S$ (chosen small enough to be disjoint from the sutures). We will make use of this viewpoint in proving Lemmas \ref{removepairedloopslemma} and \ref{removetwistslemma}.
\end{remark}

The following result will allow us to test whether various sutured manifolds are taut. If we find a suitable sequence of sutured manifold decompositions that ends at a sutured manifold we already know to be taut, it tells us that the original one is too, and moreover that in fact every sutured manifold in the sequence is taut. We will use this idea without explicit mention.

\begin{proposition}[\cite{MR723813} Lemma 3.5 (see also \cite{MR992331} Theorem 3.6)]\label{tautnessprop}
Let $(M,s)$ be a sutured manifold, and let $S$ be a connected, orientable surface properly embedded in $M$ with $\partial S$ transverse to the sutures.
Suppose that no component of $\partial S$ disjoint from $s$ bounds a disc in either $R_+$ or $R_-$, and that if $S$ is a disc then $\partial S$ intersects $s$.
Let $(M',s')$ be the result of a sutured manifold decomposition along $S$.
If $(M',s')$ is taut then so is $(M,s)$.
\end{proposition}

The next three results of Gabai help to detect whether or not a sutured manifold is a product sutured manifold. We will use them in Section \ref{fibrednesssection} to show when the Seifert surfaces we constructed in Section \ref{algorithmsection} are fibre surfaces.

\begin{proposition}[\cite{MR870705} Lemmas 2.2 and 2.5]\label{productprop}
Let $(M,s)$ be a sutured manifold, and let $S$ be either a product disc or a product annulus in $M$.
Let $(M',s')$ be the result of a sutured manifold decomposition along $S$. Then $(M',s')$ is a product sutured manifold if and only if $(M,s)$ is.
\end{proposition}

\begin{proposition}[\cite{MR870705} Lemma 2.4]\label{fibrednessprop}
If $(M, s)$ is a product sutured manifold and $(M, s)\rightsquigarrow (M', s')$ is a sutured manifold decomposition such that $(M', s')$ is taut, then $(M', s')$ is a product sutured manifold.
\end{proposition}

\begin{proposition}[\cite{MR870705} Lemma 2.7]\label{notfibredprop}
Let $(M,s)$ be a sutured manifold, and let $S$ be a connected, orientable surface properly embedded in $M$ with $\partial S$ transverse to $s$. Suppose that $S$ is neither a product disc nor a product annulus.
Let $(M_+,s_+)$ and $(M_-,s_-)$ be the results of sutured manifold decompositions along $S$ with opposite orientations on $S$. Suppose that $(M_+,s_+)$ and $(M_-,s_-)$ are both taut. Then $(M,s)$ is not a product sutured manifold.
\end{proposition}

In this paper we will follow the techniques for practical sutured manifold decomposition suggested and developed by Gabai (see for example \cite{MR823442}). In particular, as far as possible we will work with the complements of sutured manifolds that are `lying flat on the plane', meaning we can draw a nice picture, at least locally. Moreover, the decompositions used will mostly be along surfaces that are the `complementary regions' of this picture. The following lemmas are designed to help simplify such decompositions.

\begin{lemma}\label{removepairedloopslemma}
Let $(M,s)$ be a sutured manifold. Suppose there is an annulus $A$ in $\partial M$ such that $s\cap A$ consists of two simple closed curves that are core curves of $A$. Let $S$ be a surface properly embedded in $M$ such that $\partial S\cap A$ is a simple arc that is essential in $A$ and meets the two curves of $s\cap A$ each once.

Let $s'=s\setminus(s\cap A)$.
Then, up to isotopy of the sutures, sutured manifold decompositions of $(M,s)$ and $(M,s')$ along $S$ give the same result.
\end{lemma}
\begin{proof}
\begin{figure}[htbp]
\centering
\input{picturefiles/decompositionspic1.tex}
\caption{\label{decompositionspic1}}
\end{figure}
Figure \ref{decompositionspic1} shows the effect of altering the sutures in the two cases (see Remark \ref{twostepremark}). Although there are two possibilities for the result of this process, depending on the choice of orientation of $S$, we have only shown one here; the other picture is symmetric. The embedding into $\Sphere$ used in Figure \ref{decompositionspic1} has been chosen to imitate how we will use Lemma \ref{removepairedloopslemma} but is not significant for this proof.
\end{proof}

\begin{corollary}\label{toruscor}
Let $(M,s)$ be a sutured manifold such that $M$ is a solid torus and all the sutures are parallel and longitudinal. Then $(M,s)$ is taut.
\end{corollary}
\begin{proof}
Let $S$ be a $\partial$--compression disc in $M$ such that $\partial S$ intersects each component of $s$ exactly once.
Note that because $(M,s)$ is a sutured manifold, the number of sutures is necessarily even. 
We may repeatedly apply Lemma \ref{removepairedloopslemma} to reduce to the case that there are in fact no sutures. Decomposing along $S$ then gives a ball with a single suture, which is taut.
\end{proof}

\begin{remark}\label{testfibredremark}
Given a sutured manifold $(M,s)$, suppose that we find a sequence of sutured manifold decompositions of the form described in Proposition \ref{tautnessprop} that ends at a solid torus with longitudinal sutures. Since Lemma \ref{toruscor} tells us that this final sutured manifold is taut, Proposition \ref{tautnessprop} shows that every sutured manifold in the sequence (and in particular $(M,s)$) is also taut. This means we can apply Proposition \ref{fibrednessprop}. From this we see that if $(M,s)$ is a product sutured manifold then so is every other manifold in the sequence. To show that $(M,s)$ is not a product manifold, therefore, it is sufficient to find one of the manifolds in the sequence that is not. In particular, if either $R_+$ or $R_-$ is disconnected in any manifold in the sequence then $(M,s)$ is not a product sutured manifold. We will use this method in Section \ref{fibrednesssection} to prove that most of the three-bridge links we are considering are not fibred.
\end{remark}

\begin{lemma}\label{removetwistslemma}
Let $(M,s)$ be a sutured manifold. Suppose there is an annulus $A$ in $\partial M$ such that $s\cap A$ consists of two simple arcs that are essential in $A$. Let $S$ be a surface properly embedded in $M$ such that $\partial S\cap A$ is a simple arc that is essential in $A$ and meets the two arcs of $s\cap A$ minimally and in two points each.

Let $s'$ be a new set of sutures given by altering $s\cap A$ by a single Dehn twist around the core of $A$ so that each arc meets $\partial S\cap A$ only once.
Then, up to isotopy of the sutures, sutured manifold decompositions of $(M,s)$ and $(M,s')$ along $S$ (with the same orientation) give the same result.
\end{lemma}
\begin{proof}
\begin{figure}[htbp]
\centering
\input{picturefiles/decompositionspic2.tex}
\caption{\label{decompositionspic2}}
\end{figure}
See Figure \ref{decompositionspic2}. In this case both possibilities (coming from the choice of orientation) are shown as the results are different.
\end{proof}

\begin{corollary}\label{decomposeblockscor}
Let $(M,s)$ be a sutured manifold. Suppose that part of $(M,s)$ is as shown in Figure \ref{decompositionspic3}a, where the box represents $n\geq 0$ twists in the direction shown and the $k$ denotes $k\geq 0$ parallel copies of the labelled curve.
That is, a section of $(M,s)$ is made up of one copy of Figure \ref{decompositionspic4}a, $n$ copies of Figure \ref{decompositionspic4}b, $k$ copies of Figure \ref{decompositionspic4}c and one copy of Figure \ref{decompositionspic4}d.
\begin{figure}[htbp]
\centering
(a)
\input{picturefiles/decompositionspic3a.tex}
(b)
\input{picturefiles/decompositionspic3b.tex}
(c)
\input{picturefiles/decompositionspic3c.tex}
(d)
\input{picturefiles/decompositionspic3d.tex}
\caption{\label{decompositionspic3}}
\end{figure}
\begin{figure}[htbp]
\centering
(a)
\input{picturefiles/decompositionspic4a.tex}
(b)
\input{picturefiles/decompositionspic4b.tex}
(c)
\input{picturefiles/decompositionspic4c.tex}
(d)
\input{picturefiles/decompositionspic4d.tex}
(e)
\input{picturefiles/decompositionspic4e.tex}
\caption{\label{decompositionspic4}}
\end{figure}

Let $S$ be the visible disc properly embedded in $M$ (see Figure \ref{decompositionspic4}e).
If $n$ is odd then there is a choice of orientation of $S$ such that decomposing along $S$ with that orientation gives the result shown in Figure \ref{decompositionspic3}b.
If $n$ is even then we may orient $S$ such that the decomposition returns Figure \ref{decompositionspic3}c. If, additionally, $n\geq 2$ then taking the opposite orientation on $S$ changes the result to that shown in Figure \ref{decompositionspic3}d.

Note that the manifolds being decomposed are those on the outside of the picture (that is, they are drawn from a viewpoint in the interior of the manifold being decomposed).
\end{corollary}
\begin{proof}
First suppose that $n$ is odd.  As $(M,s)$ is a sutured manifold, $k$ must be even. By Lemma \ref{removepairedloopslemma} we may assume that $k=0$. Moreover, by Lemma \ref{removetwistslemma} we may further assume that either $n=1$ or $n=3$. The decompositions in these two cases are shown in Figure \ref{decompositionspic5}.
\begin{figure}[htbp]
\centering
\input{picturefiles/decompositionspic5.tex}
\caption{\label{decompositionspic5}}
\end{figure}

Now suppose instead that $n$ is even. Then $k$ is odd, so we may assume that $k=1$. Again using Lemma \ref{removetwistslemma}, we can reduce to the cases $n=0$ and $n=2$. The decompositions in these two cases are shown in Figure \ref{decompositionspic6}a.
The decomposition given by taking the opposite orientation on $S$ in the case $n=2$ is shown in Figure \ref{decompositionspic6}b.
\begin{figure}[htb]
\centering
(a)
\input{picturefiles/decompositionspic6a.tex}
(b)
\input{picturefiles/decompositionspic6b.tex}
\caption{\label{decompositionspic6}}
\end{figure}
\end{proof}

We will make repeated use of Corollary \ref{decomposeblockscor}, without explicit mention. Unless stated otherwise, in the case of an even number of twists we will use the decomposition that gives Figure \ref{decompositionspic3}c.

\begin{lemma}\label{wordmanifoldlemma}
Let $w$ be a non-empty circular word in the alphabet $\{a, b, c, d\}\cup\{ A, B, C, D, E, F\}\cup\{\alpha, \beta, \gamma, \delta\}$, obeying the following rules. 
\begin{itemize}
\item A letter $a$ or $b$ must be followed by a letter $A$, $B$ or $C$.
\item A letter $c$ or $d$ must be followed by a letter $D$, $E$ or $F$.
\item A letter $A$, $B$ or $C$ must be followed by a letter $\alpha$ or $\beta$.
\item A letter $D$, $E$ or $F$ must be followed by a letter $\gamma$ or $\delta$.
\item A letter $\alpha$ or $\gamma$ must be followed by a letter $a$ or $c$.
\item A letter $\beta$ or $\delta$ must be followed by a letter $b$ or $d$.
\end{itemize}

Form a $3$--manifold $M$ together with a collection $s$ of simple arcs and closed curves on $\partial M$ by combining the blocks shown in Figures \ref{decompositionspic7}, \ref{decompositionspic8} and \ref{decompositionspic9} in a circle according to the word $w$. Here the $o$ may be any odd number of twists, and the $e$ any even number, in the direction indicated (in blocks $C$ and $F$ the twists may be in either direction, but we may assume they are all in the same direction). The numbers on the markings denote parallel copies of the labelled arc or curve. If $(M,s)$ is a sutured manifold, then it is not taut if and only if $w$ contains only the letters $a$, $c$, $\alpha$, $\gamma$, $B$ and $E$ and every $e$ represents $0$ twists. 
\begin{figure}[bp]
\centering
\input{picturefiles/decompositionspic7.tex}
\caption{\label{decompositionspic7}}
\end{figure}
\begin{figure}[htbp]
\centering
\input{picturefiles/decompositionspic8.tex}
\caption{\label{decompositionspic8}}
\end{figure}
\begin{figure}[htbp]
\centering
\input{picturefiles/decompositionspic9.tex}
\caption{\label{decompositionspic9}}
\end{figure}
\end{lemma}
\begin{proof}
First suppose that $w$ is made up of letters from $\{a,c,\alpha,\gamma,B,E\}$, and every $e$ represents $0$ twists. We can then ignore the letters $B$ and $E$, and construct $(M,s)$ from the blocks corresponding to the word given by deleting these from $w$. Then the simple closed curve made up of the pieces shown in Figure \ref{decompositionspic10} bounds a disc in $M$ but not in $\partial M$. This disc is a compression disc for either $R_+$ or $R_-$, showing that $(M,s)$ is not taut. Now suppose instead that $w$ is not of this form. We must show that $(M,s)$ is taut.
\begin{figure}[htbp]
\centering
\input{picturefiles/decompositionspic10.tex}
\caption{\label{decompositionspic10}}
\end{figure}

We proceed by induction on the length of $w$. We will first consider the inductive step, and afterwards return to the base case.
Assume therefore that $w$ has length at least 6. Choose one letter in $w$ that either is in $\{A,C,D,F\}$ or is in $\{B,E\}$ with the $e$ representing at least two twists.
Find a three-letter subword $w'$ of $w$, not containing the chosen letter, beginning with an $a$, $b$, $c$, or $d$. There are $24$ possible such subwords, and in each case we may locally perform a sutured manifold decomposition. By symmetry we need only check half of these cases (those beginning with an $a$ or $b$). The decompositions in these $12$ cases are shown in Figures \ref{decompositionspic11} and \ref{decompositionspic12}. For each case, there are two possibilities for the letter that precedes $w'$ in $w$, and two possibilities for the following letter. It is easily checked that, for each of the $48$ resulting combinations, the new sutured manifold corresponds to a word meeting the hypotheses of Lemma \ref{wordmanifoldlemma} that is three letters shorter than $w$. Note that this shorter word is not in general the same as the word given by deleting $w'$ from $w$. However, the letters from $\{A,B,C,D,E,F\}$ do not change, so the shorter word again matches our assumptions about the form of $w$. We may therefore apply Lemma \ref{wordmanifoldlemma} to this shorter word.
\begin{figure}[htbp]
\centering
\input{picturefiles/decompositionspic11.tex}
\caption{\label{decompositionspic11}}
\end{figure}
\begin{figure}[htbp]
\centering
\input{picturefiles/decompositionspic12.tex}
\caption{\label{decompositionspic12}}
\end{figure}

Now suppose instead that $w$ has length $3$. By taking the letter $a$, $b$, $c$ or $d$ as the start of $w$, we find that $w$ is one of the $24$ possibilities for $w'$ we previously considered. Thus we may apply the same decompositions as before. It remains to check that `closing up' each of the decomposed manifolds from Figures \ref{decompositionspic11} and \ref{decompositionspic12} (that is, gluing the left-hand edge to the right-hand edge) produces a taut sutured manifold. As before, by symmetry we may focus only on those words beginning with an $a$ or $b$. The rules about the orders of letters in $w$ further reduce the number of cases to consider; $w$ cannot be the word $aA\beta$, for example.
Moreover, $w\neq aA\alpha$ because, although this word obeys the rules, the resulting manifold is never a sutured manifold.
This leaves $5$ cases. If $w\in\{aC\alpha, bA\beta, bB\beta, bC\beta\}$ then the `closed up' manifold is a solid torus with longitudinal sutures (see Figure \ref{decompositionspic13} for the case when $w=aC\alpha$). By Corollary \ref{toruscor}, this is taut.
\begin{figure}[htbp]
\centering
\input{picturefiles/decompositionspic13.tex}
\caption{\label{decompositionspic13}}
\end{figure}

The case of the word $aB\alpha$ has to be treated separately. Taking the piece of sutured manifold in Figure \ref{decompositionspic11} and `closing it up' again results in a solid torus, but the sutures are meridional. This sutured manifold is not taut, since there is a compression disc for each of $R_+$ and $R_-$. 
From our assumptions on the form of $w$, we know that the $e$ represents at least two twists. Hence, we may replace the decomposition in Figure \ref{decompositionspic11} with the alternative decomposition offered by Corollary \ref{decomposeblockscor}. This decomposition is shown in Figure \ref{decompositionspic14}. Again, `closing up' the result leaves us with a solid torus, this time with longitudinal sutures.
\begin{figure}[htbp]
\centering
\input{picturefiles/decompositionspic14.tex}
\caption{\label{decompositionspic14}}
\end{figure}
\end{proof}

\section{Tautness}\label{tautnesssection}
%master file is threebridge.tex

In this section we will show that each of the Seifert surfaces constructed in Section \ref{algorithmsection} is taut. We will do this using sutured manifold decompositions.

\begin{definition}
Given a sutured manifold $(M,s)$ embedded in $\Sphere$, the \textit{complementary sutured manifold} is the sutured manifold $(M',s')$ given by $M'=\Sphere\setminus\Int(M)$ and $s'=s$.

For a link $K$ and a Seifert surface $R_K$ for $K$, we can form a (product) sutured manifold by taking a product neighbourhood of $R_K$ in the exterior of $K$ and adding one suture for each link component. By including the link exterior into $\Sphere$, this sutured manifold has a natural embedding into $\Sphere$. We call the complementary sutured manifold to this one the \textit{complementary sutured manifold to $R_K$}.
\end{definition}

Suppose that $R_K$ is a Seifert surface for a non-split link $K$ and is not taut. Then, by definition, there is another Seifert surface $R_K'$ for $K$ with $\chi(R_K')>\chi(R_K)$. 
There is then a third Seifert surface $R_K''$ for $K$ with $\chi(R_K'')\geq \chi(R_K)$ such that $R_K''$ is disjoint on its interior from $R_K$. This was proved by Scharlemann and Thompson in \cite{MR916076} for knots, and by Kakimizu in \cite{MR1177053} for non-split links. The surface $R_K''$ can be seen as properly embedded in the complementary sutured manifold to $R_K$.
Neither $R_K$ nor $R_K'$ contains any discs or spheres, so $\chi_-(R_K'')<\chi_-(R_K)$. Thus $R_K''$ demonstrates that the complementary sutured manifold to $R_K$ is not taut.

Taking the contrapositive of this, we see that to prove that the Seifert surface $R$ we constructed in Section \ref{algorithmsection} is taut, it is sufficient to prove that the complementary sutured manifold is taut. We do this in the following proposition.

\begin{proposition}\label{tautnessproofprop}
The Seifert surface $R$ for $L_3$ constructed in Section \ref{algorithmsection} is taut.
\end{proposition}
\begin{remark}
We remind the reader that we have restricted our attention to two-bridge links of a certain form, to reduce the number of cases we must consider. No new ideas are needed to complete the more general proof, and the relevant details can be found in the appendix.
\end{remark}
\begin{proof}
We aim to show that the complementary sutured manifold $(M_R,s_R)$ to $R$ is taut. We will do this by performing sutured manifold decompositions (along surfaces meeting the conditions of Proposition \ref{tautnessprop}) until we reach a sutured manifold to which Lemma \ref{wordmanifoldlemma} can be applied.
These decompositions will take place within the blocks from which $R$ was constructed.
We will locally draw the product sutured manifold given by the section of surface of interest, and then decompose the sutured manifold outside this (following Gabai).

\step\label{decomposeunpairedstep}
On constructing the sutured manifold locally for a piece of surface given by an unpaired $\mathcal{A}$ (see Figure \ref{tautnesspic1}a), we see a product disc in $(M_R,s_R)$. The result of decomposing along it is shown in Figure \ref{tautnesspic1}b. Looking at Lemma \ref{wordmanifoldlemma}, note that this corresponds to the string $\alpha a$ in the notation of Lemma \ref{wordmanifoldlemma}.
\begin{figure}[htbp]
\centering
(a)
\input{picturefiles/tautnesspic1a.tex}\\
(b)
\input{picturefiles/tautnesspic1b.tex}
\caption{\label{tautnesspic1}}
\end{figure}

\step\label{decomposetubesstep}
We now turn our attention to the tubes created from the paired letters.
We will decompose each tube along an annulus on the inside of the tube. One boundary component of the annulus will lie on the tube of interest, and will be disjoint from the sutures. The other will lie on whatever piece of surface is outermost of those running through the middle of the chosen tube. This could be either another tube (in which case the second boundary component will also be disjoint from the sutures) or part of the annulus $\mathbb{A}$ between $L_+$ and $L_-$ (in which case the second boundary component will intersect the sutures twice). These two cases give different results. In the former case, it may be that both boundary components lie in $R_+$, both lie in $R_-$, or one lies in each. The first two of these options are symmetric and so can be considered together, but the third is different. 

Recall that, when performing a sutured manifold decomposition along a surface $S$, a product neighbourhood of $S$ is removed from the manifold. For each of the annuli we wish to decompose along, we can take this product neighbourhood to be the length of the tube the annulus lies inside. Equivalently, we can decompose along annuli at each end of the tube and then discard the piece of sutured manifold inside the tube, provided we orient these two annuli in the same direction.
Seen like this, we can do the decompositions within the blocks from the construction of $R$ at the ends of the tubes, although we must still do them in pairs because of the orientations. 
The orientations on the annuli used in the decompositions are chosen to give results that match Lemma \ref{wordmanifoldlemma}. Ultimately these choices are forced on us by what we see in Figure \ref{tautnesspic1}; choosing the opposite orientation would often result in a sutured manifold that was not taut.

Figure \ref{tautnesspic2} shows sections of the surface $R$ at the ends of a tube that is innermost (that is, one with a maximal joining word).
The central box denotes a section of surface whose form is somewhat unknown. The tube continues through this section, and nothing interesting happens `inside' the tube in this section; this is as much as we need to know to perform the decomposition.
\begin{figure}[htbp]
\centering
\input{picturefiles/tautnesspic2.tex}
\caption{\label{tautnesspic2}}
\end{figure}
The complementary sutured manifold $(M_R,s_R)$, locally around these two blocks, is given in Figure \ref{tautnesspic3}a, drawn in an alternative position to make it easier to see the annuli we wish to decompose along (which are shown in Figure \ref{tautnesspic3}b).
The annuli meet the conditions of Proposition \ref{tautnessprop}, since each boundary component is non-separating in $\partial M_R$.
\begin{figure}[htbp]
\centering
(a)
\input{picturefiles/tautnesspic3a.tex}\\
(b)
\input{picturefiles/tautnesspic3b.tex}
\caption{\label{tautnesspic3}}
\end{figure}
Our choice of orientations on the annuli depends on the relative orientations of $\mathbb{D}$ and $\mathbb{A}$ locally (that is, whether the piece of $\mathbb{D}$ on which the left-hand end\footnote{Here we mean the left-hand end `as seen from the point of view of the tube'. That is, we mean the end from which the tube propagates to the right. If the joining word contains the infinity letter, this end will actually lie to the right of the other end of the tube in the diagram we have drawn.} of the tube lies is oriented upwards or downwards; recall that $\mathbb{A}$ is oriented downwards). Figures \ref{tautnesspic4} and \ref{tautnesspic5} show the choices in these two cases, together with the results of the corresponding sutured manifold decompositions.
\begin{figure}[htbp]
\centering
\input{picturefiles/tautnesspic4.tex}
\caption{\label{tautnesspic4}}
\end{figure}
\begin{figure}[htbp]
\centering
\input{picturefiles/tautnesspic5.tex}
\caption{\label{tautnesspic5}}
\end{figure}
Figures \ref{tautnesspic6}, \ref{tautnesspic7}, \ref{tautnesspic8} and \ref{tautnesspic9} are the analogues to Figures \ref{tautnesspic2}, \ref{tautnesspic3}, \ref{tautnesspic4} and \ref{tautnesspic5} respectively in the case of a tube that is not innermost (that is, one with a joining word that is not maximal), where we see part of another tube rather than part of $\mathbb{A}$. Here the two cases come from whether the orientations of the two tubes agree or disagree.
\begin{figure}[htbp]
\centering
\input{picturefiles/tautnesspic6.tex}
\caption{\label{tautnesspic6}}
\end{figure}
\begin{figure}[htbp]
\centering
(a)
\input{picturefiles/tautnesspic7a.tex}\\
(b)
\input{picturefiles/tautnesspic7b.tex}
\caption{\label{tautnesspic7}}
\end{figure}
\begin{figure}[htbp]
\centering
\input{picturefiles/tautnesspic8.tex}
\caption{\label{tautnesspic8}}
\end{figure}
\begin{figure}[htbp]
\centering
\input{picturefiles/tautnesspic9.tex}
\caption{\label{tautnesspic9}}
\end{figure}

In each of these four cases, the two pieces of sutured manifold we see at the end of the decompositions both correspond to the string $\alpha a$ in the language of Lemma \ref{wordmanifoldlemma} (with $k,l,n\in\{0,1\}$ and $m\in\{0,1,2\}$).

\step\label{wholewordstep}
Let $(M_R',s_R')$ be the sutured manifold that results after performing on $(M_R,s_R)$ all the sutured manifold decompositions given in the first two steps.
We have shown that each letter $\mathcal{A}$ in $\mathcal{W}$ gives rise to a section of $(M_R',s_R')$ that corresponds to the string $\alpha a$ in the language of Lemma \ref{wordmanifoldlemma}. Between any consecutive pair of these sections of $(M_R',s_R')$ is a section coming from a letter $\mathcal{O}$ or $\mathcal{E}$ in $\mathcal{W}$. The infinity letter gives a section corresponding to a letter $C$, with $n=0$ if the infinity letter lies in the joining word of at least one paired $\mathcal{A}$, and $n=1$ if the infinity letter is not contained in any joining word. The number of twists, if $n=1$, will match the number of crossings between $L_+$ and $L_-$.
A letter $\mathcal{O}$ in $\mathcal{W}$ that is not the infinity letter gives a section of $(M_R',s_R')$ corresponding to an $A$, again with $n=0$ if this $\mathcal{O}$ lies in a joining word and $n=1$ otherwise. 
Similarly, an $\mathcal{E}$ that is not the infinity letter gives a section corresponding to a $B$, with $n=0$ if the $\mathcal{E}$ is in a joining word and $n=1$ else.

We conclude, therefore, that $(M_R',s_R')$ satisfies the hypotheses of Lemma \ref{wordmanifoldlemma}. The word $w_R$ describing this manifold contains a letter $C$. Hence Lemma \ref{wordmanifoldlemma} tells us that $(M_R',s_R')$ is taut.
\end{proof}

\section{Fibredness}\label{fibrednesssection}
%master file is threebridge.tex

If $R_K$ is a Seifert surface for a non-split link $K$, then $K$ is a fibred link if and only if the complementary sutured manifold to $R_K$ is a product sutured manifold.

\begin{proposition}\label{discembeddedprop}
Suppose $L_3$ is fibred. Then $\mathcal{W}$ contain no letter $\mathcal{O}$, and every $\mathcal{E}$ (other than the infinity letter) represents no twists. Equivalently, the disc $\mathbb{D}$ bounded by $L_{\mathbb{D}}$ is embedded in the projection plane.
\end{proposition}
\begin{proof}
Since $L_3$ is fibred and $R$ is a taut Seifert surface for $L_3$, the complementary sutured manifold $(M_R,s_R)$ to $R$ is a product sutured manifold. 
Given Remark \ref{testfibredremark}, this means that every sutured manifold in the sequence of decompositions in Section \ref{tautnesssection} must also be a product sutured manifold. In particular, $(M_R',s_R')$ must be a product sutured manifold (as must every step in its decomposition).
This implies that $R_+$ and $R_-$ must be connected. 

Suppose that $\mathcal{W}$ contains at least one letter $\mathcal{O}$. 
It then contains at least one pair of instances of the letter $\mathcal{A}$.
Choose a pair with a maximal joining word.
From the proof of Proposition \ref{tautnessproofprop}, we know that this pair gives rise to two sections of $(M_R',s_R')$ as shown in either Figure \ref{tautnesspic4} or Figure \ref{tautnesspic5}. Without loss of generality, assume they are as in Figure \ref{tautnesspic5}. The right-hand section gives a string $\alpha a$ in the word $w_R$ constructed in the proof of Proposition \ref{tautnessproofprop} with $k=l=1$. The next letter in $w_R$ is an $A$, $B$ or $C$, with $n=1$. Following that is an $\alpha$, also with $n=1$.
By isotoping one suture, we could alter this sequence so that the $a$ has $k=2$ and $l=0$, while the next two letters each have $n=0$. With $k=2$, we immediately see that either $R_+$ or $R_-$ is disconnected, a contradiction. Hence we find that $\mathcal{W}$ contains no letter $\mathcal{O}$.

With no $\mathcal{O}$ in $\mathcal{W}$, there also cannot be any $-$ symbols, and so there are no paired letters. Thus the Seifert surface $R$ is formed from the disc $\mathbb{D}$ and the annulus $\mathbb{A}$ by making the change depicted in Figure \ref{algorithmpic10} once for each $\mathcal{A}$ in $\mathcal{W}$.

Now suppose that there is a letter $\mathcal{E}$ in $\mathcal{W}$ that is not the infinity letter and represents at least two twists. Since the infinity letter is also an $\mathcal{E}$, we know that $\mathcal{W}$ has length at least $6$. This $\mathcal{E}$ gives us a section of $(M_R',s_R')$ corresponding to a letter $B$ in $w_R$. The $B$ is both preceded and followed in $w_R$ by the string $\alpha a$. First focus on the central string $aB\alpha$. Let $S$ be the disc that is visible when these three blocks are pictured together in order (similar to Figures \ref{decompositionspic4}e and \ref{decompositionspic11}). We wish to show that the sutured manifold we have is not taut by applying Proposition \ref{notfibredprop} to the disc $S$. Note that, because the $\mathcal{E}$ represents at least two twists, $\partial S$ intersects the sutures more than twice, and so $S$ is not a product disc.

We have already performed the decomposition along $S$ with one orientation in the proof of Proposition \ref{tautnessproofprop}, and in this case the result of the decomposition is taut. Accordingly, we now need to consider the effect of decomposing using the other orientation on $S$. In practice this means that, instead of using the decomposition shown in Figure \ref{decompositionspic11}, we make use of the alternative decomposition offered by Corollary \ref{decomposeblockscor} (which we can do because the $\mathcal{E}$ represents at least two twists).
The result of this decomposition is shown in Figure \ref{decompositionspic14}. Combining this, now, with the adjacent $a$ and $\alpha$ blocks gives the piece of sutured manifold shown in Figure \ref{fibrednesspic1}. In the language of Lemma \ref{wordmanifoldlemma}, this corresponds to $\beta b$. Thus we again arrive at a sutured manifold meeting the hypotheses of Lemma \ref{wordmanifoldlemma}. Since the corresponding word contains a $b$, Lemma \ref{wordmanifoldlemma} says that this sutured manifold is also taut.
\begin{figure}[htbp]
\centering
\input{picturefiles/fibrednesspic1.tex}
\caption{\label{fibrednesspic1}}
\end{figure}

This shows that the hypotheses of Proposition \ref{notfibredprop} hold, and so $(M_R,s_R)$ is not a product sutured manifold. This is again a contradiction. Hence every $\mathcal{E}$ in $\mathcal{W}$ other than the infinity letter represents zero twists.
\end{proof}

\begin{proposition}\label{isfibredprop}
Suppose that $\mathcal{W}$ contains no letter $\mathcal{O}$, and every $\mathcal{E}$ other than the infinity letter represents no twists. Then $L_3$ is fibred.
\end{proposition}
\begin{proof}
As in the proof of Proposition \ref{discembeddedprop}, with no $\mathcal{O}$ in $\mathcal{W}$ there are no paired letters, and $R$ is formed from the disc $\mathbb{D}$ and the annulus $\mathbb{A}$ by making the change depicted in Figure \ref{algorithmpic10} once for each $\mathcal{A}$ in $\mathcal{W}$.
We will show that the complementary sutured manifold $(M_R,s_R)$ to $R$ is a product sutured manifold by repeatedly using Proposition \ref{productprop} until we reach a manifold we can easily identify. More precisely, we will trace through the proof of Proposition \ref{tautnessproofprop} and see that, with the current restrictions in $\mathcal{W}$, every decomposition used is along a product disc.

First notice that, because there are no paired letters in $\mathcal{W}$, and therefore no tubes in $R$, every decomposition used to pass from $(M_R,s_R)$ to $(M_R',s_R')$ is as in Figure \ref{tautnesspic1}. This decomposition is along a product disc, as required. Thus, by Proposition \ref{productprop}, $(M_R,s_R)$ is a product sutured manifold if and only if $(M_R',s_R')$ is.

The word $w_R$ that describes $(M_R',s_R')$ is made up of the string $aB\alpha$ repeated $N_R$ times and the string $aC\alpha$, for some $N_R\geq 0$.
From Figure \ref{tautnesspic1}, we can see that for each $a$ in the word $w_R$ we have $k=0$ and $l=1$, while for each $\alpha$ in $w_R$ we have $m=0$ and $n=1$.
We will now show that $(M_R',s_R')$ is a product sutured manifold by induction on $N_R$.

First suppose that $N_R=0$. Then $w_R$ is $aC\alpha$, and we can obtain a picture of $(M_R',s_R')$ by `closing up' the block shown in Figure \ref{decompositionspic11}. Since $k+l+m=1$ in this case, we find that $(M_R',s_R')$ is a solid torus with two longitudinal sutures, which is a product sutured manifold.

Suppose instead that $N_R>0$. In this case we decompose the section of $(M_R',s_R')$ coming from one instance of $aB\alpha$ in $w_R$ as in Figure \ref{decompositionspic11}. As $k+l+m=1$ and the $e$ here represents zero twists, the disc we decompose along is again a product disc. After this decomposition, the resulting sutured manifold is what we would have had if the value of $N_R$ had been one lower. By induction, this is a product sutured manifold, and so Proposition \ref{productprop} again tells us that $(M_R',s_R')$ is also a product sutured manifold, as required.
\end{proof}

\begin{corollary}\label{whenfibredcor}
The three-bridge link $L_3$ is fibred if and only if, in the continued fraction $[a_1,a_2,a_3,\ldots,a_N]$ of the two-bridge link $L_2$, either $N=1$ and $a_1$ is even, or $N>1$ and both $a_1$ and $a_N$ are odd while $a_i$ is even for $1<i<N$.
\end{corollary}
\begin{proof}
Combining Propositions \ref{discembeddedprop} and \ref{isfibredprop}, we see that $L_3$ is fibred if and only if the disc $\mathbb{D}$ is embedded in the projection plane. By examining Figure \ref{algorithmpic1} we find that this is the case exactly in the situation described.
\end{proof}

\section{Satellites}\label{satellitesection}
%master document is threebridge.tex

We now turn to considering certain satellite links, in the direction of the work of Hirasawa and Murasugi in \cite{zbMATH05820036}.

Let $V$ be a solid torus, and $L\subseteq \Int(V)$ a link that does not lie within a ball in $V$ and is not a core curve of $V$. 
Choose a longitude $\lambda$ on the torus $\partial V$, and let $l$ be the winding number of $L$ in $V$ (that is, $[L]=l[\lambda]\in H_1(V)$).
Let $K_c\subseteq\Sphere$ be a knot other than the unknot, and let $R_c\subseteq\Sphere\setminus\nhd(K_c)$ be a (taut) Seifert surface for $K_c$.

\begin{definition}
A \textit{satellite link $\hat{L}$ with pattern $(V,L,\lambda)$ and companion $K_c$} is the image of $L$ in $\Sphere$ under a homeomorphism $i\colon V\to\nhd(K_c)\subseteq\Sphere$ that sends $\lambda$ to $\partial R_c$.
\end{definition}

\begin{definition}
Say that a surface $R\subseteq V$ is an \textit{oriented spanning surface for the pattern $(V,L,\lambda)$} if $R$ is oriented and $\partial R$ consists of $L$ (as an oriented link) together with some simple closed curves on $\partial V$ parallel to $\lambda$. 

Say that $R$ is \textit{taut} if it has maximal Euler characteristic among oriented spanning surfaces for $(V,L,\lambda)$.
If $R$ is taut, set $\chi(V,L,\lambda)=\chi(R)$.
\end{definition}

Note that, by considering homology, we know $|R\cap\partial V|\geq|l|$.

As for a Seifert surface in $\Sphere$, we can use $R$ to form a sutured manifold by removing a product neighbourhood of $R$ from $V$ and dividing up the boundary of the resulting manifold according to the orientation of $R\cup\partial V$. 
We call the resulting sutured manifold the \textit{complementary sutured manifold to $R$ in $V$}. This enables us to make the following definition.

\begin{definition}\label{patternfibreddefn}
The pattern $(V,L,\lambda)$ is \textit{fibred} if there is an oriented spanning surface $R_{\lambda}\subseteq V$ such that $|R_{\lambda}\cap\partial V|=|l|$ and the complementary sutured manifold to $R_{\lambda}$ in $V$ is a product sutured manifold.
\end{definition}

Note that our definition of the complementary sutured manifold calls for an orientation on $\partial V$. 
For the purposes of determining whether a pattern is fibred it does not matter which orientation is chosen, provided the whole of $\partial V$ is oriented the same way.  The condition $|R_{\lambda}\cap\partial V|=|l|$ means that all the curves of $R_{\lambda}\cap\partial V$ are oriented the same way.
If $|l|\neq 0$ then the two choices of orientation of $\partial V$
 will in fact yield the same sutured manifold. If $|l|=0$ then $(V,L,\lambda)$ can never be fibred, since the sutured manifold has disconnected boundary.

The following result shows that we can, if desired, drop the condition $|R_{\lambda}\cap\partial V|=|l|$ from Definition \ref{patternfibreddefn}.

\begin{lemma}
Let $R$ be an oriented spanning surface for $(V,L,\lambda)$. Then there is an oriented spanning surface $R'$ with $\chi(R)=\chi(R')$ and $|R'\cap\partial V|=|l|$. Moreover, $R'$ can be chosen such that the complementary sutured manifold to $R$ is a product sutured manifold if and only if the same is true for $R'$.
\end{lemma}
\begin{proof}
Without loss of generality, give $\partial V$ an orientation pointing into $V$.
Recall that $|R\cap\partial V|\geq|l|$. If $|R\cap\partial V|>|l|$ then there exist two curves $\lambda_1$ and $\lambda_2$ of $R\cap\partial V$ that are adjacent in $\partial V$ such that the orientations of these surfaces are locally as shown in Figure \ref{satellitepic1}a.
Surgery on $R$ along the sub-annulus of $\partial V$ between $\lambda_1$ and $\lambda_2$, and otherwise disjoint from $R$, gives a new surface $R''\subset V$ with $\partial R''=\partial R\setminus(\lambda_1\cup\lambda_2)$ (plus a boundary-parallel annulus that we discard). Note that $\chi(R'')=\chi(R)$.
\begin{figure}[htbp]
\centering
(a)
\input{picturefiles/satellitepic1a.tex}
(b)
\input{picturefiles/satellitepic1b.tex}
\caption{\label{satellitepic1}}
\end{figure}

Let $(M,s)$ and  $(M'',s'')$ be the complementary sutured manifolds to $R$ and $R''$ respectively. 
There is a product annulus $S$ in $(M'',s'')$ with one boundary component on $R''$ and the other on $\partial V$, as shown in Figure \ref{satellitepic1}(b).
Performing a sutured manifold decomposition along this product annulus results in the sutured manifold $(M,s)$. Therefore, by Proposition \ref{productprop}, $(M'',s'')$ is a product sutured manifold if and only if $(M,s)$ is.

Thus, by induction, we can repeat this process until we arrive at an oriented spanning surface that meets $\partial V$ in exactly $|l|$ curves.
\end{proof}

Note that if we had instead chosen to orient $\partial V$ pointing out of $V$ then the proof would have been the same except for reversing the arrows in Figure \ref{satellitepic1}.
Another option would have been to orient $\partial V$ differently when constructing the complementary sutured manifolds to $R$ and $R'$. However, we have already seen that the orientation on $\partial V$ does not affect whether the complementary sutured manifold to $R'$ is a product sutured manifold (although there is no a priori reason for this to be true for $R$).

\begin{theorem}[\cite{zbMATH05545482} Theorem 1]
The satellite link $\hat{L}$ is fibred if and only if both the companion knot $K_c$ and the pattern $(V,L,\lambda)$ are fibred.
\end{theorem}

The proofs in \cite{zbMATH05545482} work from an equivalent definition of when a pattern is fibred given in terms of fundamental groups. It is noted in the paper that there exists an alternative proof of \cite{zbMATH05545482} Theorem 2 (see Theorem \ref{addmeridiansthm} below) via sutured manifold theory.
Using these ideas yields some useful information about the fibre surfaces. 
Let $T$ be the torus $\partial\nhd(K_c)=\partial i(V)$, and choose a taut Seifert surface $R_{\hat{L}}$ for $\hat{L}$ that has minimal intersection with $T$ among such surfaces. Then $R_{\hat{L}}\cap T$ will be $|l|$ simple closed curves parallel to $\lambda$, all oriented in the same direction.

\begin{remark}
If instead we had taken $K_c$ to be the unknot, this statement would not hold. While it would be possible to arrange that $|R_{\hat{L}}\cap T|=|l|$, it might be necessary to increase the genus of $R_{\hat{L}}$ in order to do so. This is why some patterns are not fibred even though within $\Sphere$ they give fibred links. However, if there were a taut Seifert surface for $R_{\hat{L}}$ that could be made to intersect $T$ in only $|l|$ simple closed curves then the pattern would be fibred if and only if the link in $\Sphere$ was.
\end{remark}

If $l=0$ then $R_{\hat{L}}\cap T=\emptyset$. Because $T$ is essential in $\Sphere\setminus\nhd(\hat{L})$ and there can be no essential torus within a product manifold, this implies that $\hat{L}$ is not fibred. Note that here we have used that $K_c$ is not the unknot. Also, $(V,L,\lambda)$ is not fibred in this case, as previously noted.
If instead $l\neq 0$ then $R_{\hat{L}}\cap T$ divides $R_{\hat{L}}$ into $|l|$ taut Seifert surfaces for $K_c$ (possibly with orientation reversed) and an oriented spanning surface for $(V,L,\lambda)$. If $\hat{L}$ is fibred then this oriented spanning surface demonstrates that the pattern $(V,L,\lambda)$ is also fibred.

We can also use this description to understand the genus of a satellite knot. More precisely, we see that 
\[
\chi(\hat{L})=|l|\chi(K_c)+\chi(V,L,\lambda),
\]
where $\chi(L')$ denotes the maximal Euler characteristic of a Seifert surface for a link $L'$.

Theorem 2 of \cite{zbMATH05545482} gives a means of testing whether the pattern $(V,L,\lambda)$ is fibred.
Let $V'$ be an unknotted solid torus in $\Sphere$. Let $\lambda'$ be a simple closed curve on the torus $\partial V'$ that bounds a disc in the complementary solid torus $\Sphere\setminus\Int(V')$.
Embed $V$ in $\Sphere$ by a homeomorphism $j\colon V\to V'$ that sends $\lambda$ to $\lambda'$.
Let $\mu$ and $\mu'$ be disjoint simple closed curves on $\partial V'$ that bound discs in $V$.
Denote by $L_3$ the link $j(L)\cup\mu\cup\mu'$, where $\mu$ and $\mu'$ are given opposite orientations.

\begin{theorem}[\cite{zbMATH05545482} Theorem 2]\label{addmeridiansthm}
The pattern $(V,L,\lambda)$ is fibred if and only if the link $L_3$ is fibred.
\end{theorem}

Again, by considering sutured manifolds we gain more precise information. Push the two parallel link components $\mu$ and $\mu'$ off the torus $\partial V'$ into the solid torus $\Sphere\setminus\Int(V')$, of which each is a core curve. Together $\mu$ and $\mu'$ span an annulus $\mathbb{A}$ within this solid torus. Choose a taut Seifert surface $R_3$ for $L_3$ with minimal intersection with $\partial V'$. Then $|R_3\cap\partial V'|=|l|$ and each curve of $R_3\cap\partial V'$ bounds a meridian disc of $\Sphere\setminus\Int(V')$. Moreover, $R_3$ is divided by $\partial V'$ into two sections, one an oriented spanning surface $R_{\lambda}$ for $(V,L,\lambda)$ and the second an oriented spanning surface $R_{\mathbb{A}}$ for the pattern $(\Sphere\setminus\Int(V'),\partial \mathbb{A},\lambda)$.
The Seifert surface $R_3$ can be chosen such that $R_{\mathbb{A}}$ is the double curve sum of the annulus $\mathbb{A}$ and $|l|$ meridian discs for $\Sphere\setminus\Int(V')$ (punctured by $\partial \mathbb{A}$). The complementary sutured surface to $R_{\mathbb{A}}$ in $\Sphere\setminus\Int(V')$ is then fibred if and only if $l\neq 0$.
In addition, we see that 
\[
\chi(L_3)=\chi(R_3)=\chi(R_{\lambda})+\chi(R_{\mathbb{A}})=\chi(V,L,\lambda)-|l|.
\]\

In \cite{zbMATH05820036}, Hirasawa and Murasugi consider satellite knots where the pattern is a torti-rational knot. Any such pattern can be constructed as follows. Choose a two-component two-bridge link. A neighbourhood of one component is an unknotted solid torus in $\Sphere$, with the second component lying in the complementary solid torus. Take $V$ to be this second solid torus, and $L$ the second component of the two-bridge link. In addition, take $\lambda$ to be a longitude on $\partial V$ that \emph{does not} bound a disc in the complement of $V$.

For a satellite knot $\hat{K}$ with a pattern of this form and a fibred companion knot, Hirasawa and Murasugi give the genus of $\hat{K}$ and when it is fibred, in terms of the coefficients of a particular continued fraction expansion for the original two-bridge link and the choice of longitude $\lambda$. When $l\neq 0$ these conditions are also related to the Alexander polynomial as follows.

\begin{theorem}[\cite{zbMATH05820036} Theorem 13.1]\label{nonzerowindingthm}
Assume $l\neq 0$. The genus of $\hat{K}$ is half the degree of the Alexander polynomial of $\hat{K}$. In addition, $\hat{K}$ is fibred if and only if the Alexander polynomial is monic.
\end{theorem}

In \cite{zbMATH05820036} Remark 13.3, the authors mention that they do not address the case where $\lambda$ is chosen to bound a disc in the complement of $V$. By combining Proposition \ref{tautnessproofprop} and Corollary \ref{whenfibredcor} with Theorem \ref{addmeridiansthm}, we can now complete the picture by considering this case.

Let $(V,L_{\mathbb{D}},\lambda_{\mathbb{D}})$ be the pattern produced from the two-component two-bridge link $L_2$, taking $\lambda_{\mathbb{D}}$ to bound a disc in the complement of $V$. 
When viewed within $\Sphere$, the knot $L_{\mathbb{D}}$ is the unknot. Thus the Alexander polynomial of a satellite knot $\hat{K}_{\mathbb{D}}$ constructed with this pattern is $\Delta_{K_c}(t^l)$, where $l$ is the winding number of the pattern and $\Delta_{K_c}(t)$ is the Alexander polynomial of the companion knot $K_c$ (see, for example, \cite{MR1472978} Theorem 6.15). Assuming that $K_c$ is fibred, the Alexander polynomial  of the satellite will therefore always be monic. On the other hand we will see that the satellite knot is sometimes fibred and sometimes not. We will also see that the part of Theorem \ref{nonzerowindingthm} relating to genus does not hold in this situation.

By Theorem \ref{addmeridiansthm}, $(V,L_{\mathbb{D}},\lambda_{\mathbb{D}})$ is fibred if and only if the three-bridge link $L_3$ formed in Section \ref{algorithmsection} is fibred, where $L_3$ is constructed by giving the annulus $\mathbb{A}$ the blackboard framing (that is, by creating the diagram of $L_3$ with no crossings between $L_+$ and $L_-$). 
From this we immediately have the following corollary of Corollary \ref{whenfibredcor}.

\begin{corollary}
A satellite knot with pattern $(V,L_{\mathbb{D}},\lambda_{\mathbb{D}})$ as described and fibred companion knot is fibred if and only if, in the continued fraction $[a_1,a_2,a_3,\ldots,a_N]$ of the two-bridge link $L_2$, either $N=1$ and $a_1$ is even, or $N>1$ and both $a_1$ and $a_N$ are odd while $a_i$ is even for $1<i<N$.
\end{corollary}

For $l\neq 0$, Hirasawa and Murasugi find the genus of each knot of interest by finding a Seifert surface with a given genus, and showing that this matches the lower bound on genus given by the degree of the Alexander polynomial. When $l=0$, this lower bound is not sufficient, and they use sutured manifold decompositions to prove that the surface constructed is taut, as we have done in this paper.
Unsurprisingly, the surface constructions we have used are very similar to those in \cite{zbMATH05820036}. The disc $\mathbb{D}$ is the same as the `primitive spanning disc' in \cite{zbMATH05820036} Section $3.5$. Moreover, Lemma \ref{pairinglemma} is essentially the same as the induction in \cite{zbMATH05820036}  Section 7, and the surface we arrive at in the middle of Section \ref{algorithmsection} Step \ref{tubingstep} is the `canonical surface' they construct. Note that this surface intersects a neighbourhood of $L_{\mathbb{A}}$ in $|l|$ meridian discs. In the second half of Section \ref{algorithmsection} Step \ref{tubingstep}, we take the double curve sum of this surface with the annulus $\mathbb{A}$, to arrive at the surface $R_3$ for $L_3$. 

Looking again at Euler characteristic, under the assumption $l\geq 0$ and given Remark \ref{genusremark}, we find that 
\[
\chi(V,L_{\mathbb{D}},\lambda_{\mathbb{D}})=\chi(R_3)+|l|=1-2N_+ +l=1-2N_+ +(N_+-N_-)=1-N_+-N_-.
\]
Here $N_+$ and $N_-$ are as in Section \ref{algorithmsection}, and we have used that $l=N_+-N_-$. Hence if $\hat{L}$ is a satellite with pattern $(V,L_{\mathbb{D}},\lambda_{\mathbb{D}})$ and companion $K_c$ then 
\[
\chi(\hat{L})=\chi(V,L_{\mathbb{D}},\lambda_{\mathbb{D}})+l\chi(K_c)=1-(N_++N_-)+(N_+-N_-)\chi(K_c).
\]

We will now finish with two examples to contrast with Theorem \ref{nonzerowindingthm}. Take $K_c$ to be the trefoil, which is fibred. Consider the two two-bridge links $L_{\alpha}$ and $L_{\beta}$ shown in Figure \ref{satellitepic2}a and Figure \ref{satellitepic2}b respectively.
Using these as described above to form satellites with companion $K_c$, we arrive respectively at the knots $\hat{L}_{\alpha}$ and $\hat{L}_{\beta}$ in Figure \ref{satellitepic3}.
\begin{figure}[htbp]
\centering
(a)
\input{picturefiles/satellitepic2a.tex}
(b)
\input{picturefiles/satellitepic2b.tex}
\caption{\label{satellitepic2}}
\end{figure}
\begin{figure}[htbp]
\centering
(a)
\input{picturefiles/satellitepic3a.tex}
(b)
\input{picturefiles/satellitepic3b.tex}
\caption{\label{satellitepic3}}
\end{figure}

The Alexander polynomial of both of these knots is $t^4-t^2+1$.  
However, we have shown that $\hat{L}_{\alpha}$ is fibred whereas $\hat{L}_{\beta}$ is not. Similarly, $\chi(\hat{L}_{\alpha})=1-(2+0)+(2-0)\chi(K_c)=-3$ so $g(\hat{L}_{\alpha})=2$, whereas $\chi(\hat{L}_{\beta})=1-(3+1)+(3-1)\chi(K_c)=-5$ so $g(\hat{L}_{\beta})=3$. Thus the Alexander polynomial is unable to detect either genus or fibredness in these cases.

\appendix
\section{The general case}
%master file is threebridge.tex

In Section \ref{algorithmsection}, we gave the construction of the Seifert surface $R$ for the three-bridge link $L_3$ under the assumption that the word $\mathcal{W}$ constructed contains no letters $\mathcal{B}$, $\mathcal{C}$ or $\mathcal{D}$. The same is true of the proofs of Proposition \ref{tautnessproofprop} in Section \ref{tautnesssection} and Propositions \ref{discembeddedprop} and \ref{isfibredprop} in Section \ref{fibrednesssection}. Here we give the additional details needed for the proofs without this condition. As previously mentioned, we note that no new ideas are needed, there are only more cases to check.

\paragraph{Section \ref{algorithmsection} Step \ref{pairingstep}}
Pairing of the letters in $\mathcal{W}$ should be performed as if each $\mathcal{B}$, $\mathcal{C}$ or $\mathcal{D}$ were an $\mathcal{A}$.

\paragraph{Section \ref{algorithmsection} Step \ref{tubingstep}}
To create the Seifert surface $R$ from the disc $\mathbb{D}$ and the annulus $\mathbb{A}$, tube together two sections of $\mathbb{D}$ for each pair of letters as in Section \ref{algorithmsection}.
This leaves one arc of intersection for each unpaired $\mathcal{A}$, $\mathcal{B}$, $\mathcal{C}$ or $\mathcal{D}$.
For each unpaired $\mathcal{A}$, make the change shown in Figure \ref{algorithmpic10}, as before. The analogous pictures for an unpaired $\mathcal{B}$, $\mathcal{C}$ and $\mathcal{D}$ are shown in Figures \ref{appendixpic1}, \ref{appendixpic2} and \ref{appendixpic3} respectively.
\begin{figure}[htbp]
\centering
(a)
\input{picturefiles/appendixpic1a.tex}
(b)
\input{picturefiles/appendixpic1b.tex}
\caption{\label{appendixpic1}}
\end{figure}
\begin{figure}[htbp]
\centering
(a)
\input{picturefiles/appendixpic2a.tex}
(b)
\input{picturefiles/appendixpic2b.tex}
\caption{\label{appendixpic2}}
\end{figure}
\begin{figure}[htbp]
\centering
(a)
\input{picturefiles/appendixpic3a.tex}
(b)
\input{picturefiles/appendixpic3b.tex}
\caption{\label{appendixpic3}}
\end{figure}

\paragraph{Proposition \ref{tautnessproofprop} Step \ref{decomposeunpairedstep}}
The sutured manifold decomposition of a section of the sutured manifold $(M_R,s_R)$ coming from an unpaired $\mathcal{A}$ is shown in Figure \ref{tautnesspic1}. The resulting section of $(M_R',s_R')$ corresponds to the string $\alpha a$ in the language of Lemma \ref{wordmanifoldlemma}.
The decomposition in the case of an unpaired $\mathcal{B}$ is symmetric, and results in a section of $(M_R',s_R')$ that corresponds to the string $\gamma c$ (with $k=m=0$ and $l=n=1$).
Figure \ref{appendixpic4} is the analogue to Figure \ref{tautnesspic1} in the case of an unpaired $\mathcal{C}$. The resulting section of $(M_R',s_R')$ corresponds to the string $\alpha c$. The case for a $\mathcal{D}$ is symmetric to this, and the corresponding string is $\gamma a$.
\begin{figure}[htbp]
\centering
(a)
\input{picturefiles/appendixpic4a.tex}
(b)
\input{picturefiles/appendixpic4b.tex}
\caption{\label{appendixpic4}}
\end{figure}

\paragraph{Proposition \ref{tautnessproofprop} Step \ref{decomposetubesstep}}
When we consider the sutured manifold decompositions at the ends of the tubes in $R$ coming from paired letters in $\mathcal{W}$, the addition of the letters $\mathcal{B}$, $\mathcal{C}$ and $\mathcal{D}$ introduces many new cases to be considered, as now the left-hand and right-hand ends of a tube can each correspond to one of the four letters. Rather than work through all these cases, we will work separately with the left-hand and right-hand ends of each tube. We can do this because the only interaction between the two ends is in the choice of orientation on the annuli we are decomposing along; we will make the same choices as when each end corresponds to an $\mathcal{A}$.

As before, we must consider separately the situations for a tube with a maximal joining word and for a tube with a joining word that is not maximal. First assume the joining word is maximal.

At the left-hand end of the tube, for a letter $\mathcal{A}$, the section of $(M_R',s_R')$ we see (after the decomposition) corresponds to the string $\alpha a$, either with $k=l=m=0$ and $n=1$ or with $k=l=0$ and $m=n=1$. The case for a $\mathcal{B}$ is symmetric, and gives the string $\gamma c$, again with $k=l=0$, $n=1$ and $m\in\{0,1\}$.
Figures \ref{appendixpic5}a and \ref{appendixpic6}a show the two decompositions in the case of a letter $\mathcal{C}$. The resulting section of $(M_R',s_R')$ corresponds to the string $\alpha c$
with $k=l=0$, $n=1$ and $m\in\{0,1\}$.
The situation for a $\mathcal{D}$ is symmetric to this, and the result corresponds to the string $\gamma a$
with $k=l=0$, $n=1$ and $m\in\{0,1\}$.
\begin{figure}[htbp]
\centering
(a)
\input{picturefiles/appendixpic5a.tex}
(b)
\input{picturefiles/appendixpic5b.tex}
\caption{\label{appendixpic5}}
\end{figure}
\begin{figure}[htbp]
\centering
(a)
\input{picturefiles/appendixpic6a.tex}
(b)
\input{picturefiles/appendixpic6b.tex}
\caption{\label{appendixpic6}}
\end{figure}

At the right-hand end of the tube, a letter $\mathcal{A}$ again gives a section of $(M_R',s_R')$ corresponding to the string $\alpha a$ (now with $m=n=0$, $l=1$ and $k\in\{0,1\}$). As before, the case of a letter $\mathcal{B}$ is symmetric, and the corresponding string is $\gamma c$. The two decompositions in the case of a $\mathcal{C}$ are shown in Figures \ref{appendixpic5}b and \ref{appendixpic6}b. The corresponding string is, again, $\alpha c$. By symmetry, a $\mathcal{D}$ yields the string $\gamma a$.

Now suppose instead that the joining word is not maximal. The analogues of Figures \ref{appendixpic5} and \ref{appendixpic6} in this case are Figures \ref{appendixpic7} and \ref{appendixpic8} respectively. 
Once again we find that the letters $\mathcal{A}$, $\mathcal{B}$, $\mathcal{C}$ and $\mathcal{D}$ yield the strings $\alpha a$, $\gamma c$, $\alpha c$ and $\gamma a$ respectively, with varying values of $k$, $l$, $m$ and $n$.
\begin{figure}[htbp]
\centering
(a)
\input{picturefiles/appendixpic7a.tex}
(b)
\input{picturefiles/appendixpic7b.tex}
\caption{\label{appendixpic7}}
\end{figure}
\begin{figure}[htbp]
\centering
(a)
\input{picturefiles/appendixpic8a.tex}
(b)
\input{picturefiles/appendixpic8b.tex}
\caption{\label{appendixpic8}}
\end{figure}

\paragraph{Proposition \ref{tautnessproofprop} Step \ref{wholewordstep}}
Between the sections of $(M_R',s_R')$ coming from the instances of the letters $\mathcal{A}$, $\mathcal{B}$, $\mathcal{C}$ and $\mathcal{D}$ in $\mathcal{W}$ there are sections coming from instances of the letters $\mathcal{O}$ and $\mathcal{E}$. The infinity letter gives a section corresponding to a letter $C$ in the language of Lemma \ref{wordmanifoldlemma}. An $\mathcal{O}$ that is not the infinity letter yields either an $A$ or a $D$, and an $\mathcal{E}$ that is not the infinity letter yields either a $B$ or an $E$. Thus, as before, $(M_R',s_R')$ satisfies the hypotheses of Lemma \ref{wordmanifoldlemma} and, as $w_R$ contains a letter $C$, is taut.

\paragraph{Proposition \ref{discembeddedprop}}
Suppose that the word $\mathcal{W}$ contains a letter $\mathcal{O}$. We wish to show that $(M_R',s_R')$ is not a product sutured manifold, which we do by demonstrating that at least one of $R_+$ and $R_-$ is disconnected. To do this, we consider the pieces of $(M_R',s_R')$ coming from the endpoints of a tube in $R$ corresponding to a maximal joining word.
Recall that, in Section \ref{tautnesssection}, there were two ways we decomposed $(M_R,s_R)$ given such a tube (shown in Figures \ref{tautnesspic4} and \ref{tautnesspic5}). The two cases are given by whether the left-hand or the right-hand end of the tube lies on a piece of the disc $\mathbb{D}$ that is oriented upwards (Figure \ref{tautnesspic4}) or downwards (Figure \ref{tautnesspic5}). Recall that the annulus $\mathbb{A}$ is always oriented downwards, whereas the two endpoints of the tube lie on two pieces of $\mathbb{D}$ with opposite orientations. Equivalently, we can ask whether, of the paired $+\mathcal{A}+$ and $-\mathcal{A}-$ in $\mathcal{W}$, the left-hand or the right-hand end of the tube corresponds to the $-\mathcal{A}-$. After decomposition, each end gives a section of $(M_R',s_R')$ corresponding to the string $\alpha a$ in the language of Lemma \ref{wordmanifoldlemma}. From Figure \ref{tautnesspic4} we see that if the $-\mathcal{A}-$ is at the right-hand end of the tube then the corresponding $a$ piece of $(M_R',s_R')$ has $k=l=1$. On the other hand, if the $-\mathcal{A}-$ is at the left-hand end of the tube then the corresponding $\alpha$ piece has $m=n=1$.
By symmetry we know that if the right-hand end of the tube corresponds to a $-\mathcal{B}-$ then the piece of $(M_R',s_R')$ there corresponds to the string $\gamma c$ with $k=l=1$, whereas if the left-hand end corresponds to $-\mathcal{B}-$ then we get the string $\gamma c$ with $m=n=1$. Rephrasing this observation, when considering the letters $\mathcal{A}$ and $\mathcal{B}$, if the right-hand end of the tube lies on a piece of $\mathbb{D}$ that is oriented downwards then we have an $a$ or $c$ in $w_R$ with $k=l=1$, whereas if it is the left-hand end of the tube that does so then we have an $\alpha$ or $\gamma$ in $w_R$ with $m=n=1$.
From Figures \ref{appendixpic5} and \ref{appendixpic6} we see that the same is true when considering the letters $\mathcal{C}$ and $\mathcal{D}$.

Return now to considering the tube with a maximal joining word that we have chosen. Without loss of generality, we may assume that the right-hand end of the tube lies on a piece of $\mathbb{D}$ that is oriented downwards. Then we find that $w_R$ contains a letter $a$ or $c$ with $k=l=1$. The next letter in $w_R$ is in $\{A,B,C,D,E\}$ with $n=1$. Following this is a letter $\alpha$ or $\gamma$ with $n=1$. As before, by isotoping one suture we can arrange that the $a$ or $c$ instead has $k=2$ and $l=0$. This shows that either $R_+$ or $R_-$ is disconnected.
Thus we see that $\mathcal{W}$ does not contain a letter $\mathcal{O}$.

Next suppose that $\mathcal{W}$ contains a letter $\mathcal{E}$ that is not the infinity letter and represents at least two twists. This letter $\mathcal{E}$ corresponds to a letter $B$ or $E$ in $w_R$; by symmetry we may assume it is a $B$. This $B$ is preceded in $w_R$ by an $a$ and followed by an $\alpha$. From here the argument proceeds as before.

\paragraph{Proposition \ref{isfibredprop}}
As before, we will show that $(M_R,s_R)$ is a product sutured manifold by checking that all decompositions used in the proof of Proposition \ref{tautnessproofprop} in this case are along product discs, until we reach a manifold that is clearly a product sutured manifold.

From Figures \ref{tautnesspic1} and \ref{appendixpic4} we see that every decomposition from $(M_R,s_R)$ to $(M_R',s_R')$ is along a product disc. Therefore it suffices to show that $(M_R',s_R')$ is a product sutured manifold. We will do this by induction on the length of the word $w_R$. Note that $w_R$ contains no letter $b$, $d$, $A$, $D$, $F$, $\beta$ or $\delta$, and it contains exactly one $C$. It follows that $w_R$ can be broken down into strings in $\{aB\alpha, aC\alpha, cE\gamma\}$. 

First suppose that $w_R$ has length $3$. Then $w_R$ is $aC\alpha$. We have already seen that this is a product sutured manifold.
Now suppose that $w_R$ has length at least $6$. Then it contains at least one of the other two strings. As before, if $aA\alpha$ is a subword of $w_R$ then we can perform a sutured manifold decomposition along a product disc that has the effect of deleting this subword from $w_R$. By symmetry, the same is true if $cE\gamma$ is a subword of $w_R$. Hence, by induction, we find that $(M_R',s_R')$ is a product sutured manifold.

%------------------
\bibliography{threebridgerefs}
\bibliographystyle{hplain}
%------------------

\bigskip
\noindent
Department of Physics and Mathematics

\noindent
University of Hull

\noindent
Hull, HU6 7RX, UK 

\noindent
\textit{jessica.banks[at]lmh.oxon.org}

\end{document}